\documentclass[a4paper,11pt]{article}

\usepackage[T1]{fontenc}
\usepackage[utf8]{inputenc}
\usepackage{tikz}
\usepackage{amsmath}
\usepackage{enumerate}
\usepackage{amsthm}
\usepackage{mathtools}
\usepackage{a4wide}
\usepackage{amssymb}
\usepackage{thmtools,thm-restate}
\usepackage{authblk}

\declaretheorem{theorem}
\declaretheorem[sibling=theorem]{lemma}
\declaretheorem[sibling=theorem]{claim}
\declaretheorem[sibling=theorem]{corollary}
\declaretheorem[sibling=theorem,style=definition]{definition}

\DeclarePairedDelimiter{\ceil}{\lceil}{\rceil}
\DeclarePairedDelimiter{\floor}{\lfloor}{\rfloor}

\DeclareMathOperator{\pc}{pc}
\DeclareMathOperator{\col}{col}
\DeclareMathOperator{\Col}{Col}
\newcommand{\eps}{\varepsilon}
\newcommand{\NN}{\mathbf N}
\renewcommand{\S}{\mathbf S}

\newcommand{\F}{\mathcal F}
\newcommand{\K}{\mathcal K}
\newcommand{\E}{\mathcal E}
\newcommand{\kg}{\mathcal{KG}}

\begin{document}
\title{\bfseries Vertex covering with monochromatic pieces of few colours}
\author{Marlo Eugster}
\affil{\normalsize\itshape Department of Computer Science, ETH Zurich,
Switzerland}
\author{Frank Mousset\thanks{E-mail: \texttt{moussetfrank@gmail.com}\\
Research supported in part by ISF grants 1028/16
and 1147/14, and ERC Starting Grant 633509.}}
\affil{\normalsize\itshape Institute of Mathematical Sciences, Tel Aviv University, Tel Aviv, Israel}
\date{\normalsize\today}

\maketitle

\begin{abstract}
In 1995, Erd\H{o}s and Gyárfás proved that in every $2$-colouring of the edges of $K_n$, there is a vertex covering by $2\sqrt{n}$ monochromatic paths of the same colour, which is optimal up to a constant factor. The main goal of this paper is to study the natural multi-colour generalization of this problem: given two positive integers $r,s$, what is the smallest number $\pc_{r,s}(K_n)$ such that in every colouring of the edges of $K_n$ with $r$ colours, there exists a vertex covering of $K_n$ by $\pc_{r,s}(K_n)$ monochromatic paths using altogether at most $s$ different colours?

For fixed integers $r>s$ and as $n\to\infty$, we prove that $\pc_{r,s}(K_n) = \Theta(n^{1/\chi})$, where $\chi=\max{\{1,2+2s-r\}}$ is the chromatic number of the Kneser graph $\kg(r,r-s)$. More generally, if one replaces $K_n$ by an arbitrary $n$-vertex graph with fixed independence number $\alpha$, then we have $\pc_{r,s}(G) = O(n^{1/\chi})$, where this time around $\chi$ is the chromatic number of the Kneser hypergraph $\kg^{(\alpha+1)}(r,r-s)$. This result is tight in the sense that there exist graphs with independence number $\alpha$ for which $\pc_{r,s}(G) = \Omega(n^{1/\chi})$. This is in sharp contrast to the case $r=s$, where it follows from a result of Sárközy (2012) that $\pc_{r,r}(G)$ depends only on $r$ and $\alpha$, but not on the number of vertices. 

We obtain similar results for the situation where instead of using paths, one wants to cover a graph with bounded independence number by monochromatic cycles, or a complete graph by monochromatic $d$-regular graphs.
\end{abstract}

\section{Introduction}

Call a subgraph of an edge-coloured graph \emph{monochromatic} if all its edges
have the same colour. This paper is concerned with the general problem of
covering all the vertices of an edge-coloured graph by monochromatic pieces. To
be more precise, suppose that $\F$ is a fixed family of graphs, containing the
`pieces' that we can use for the covering. A \emph{monochromatic $\F$-covering} of
an edge-coloured graph $G$ is then a collection of monochromatic subgraphs of
$G$ covering all the vertices, such that every subgraph in the collection is
isomorphic to one of the graphs in $\F$. Typical choices for $\F$ include the
the collection $\F_p$ of all paths or the collection $\F_c$ of all cycles,
where it is customary to consider single vertices and edges as degenerate
cycles.
Given a graph $G$, we are interested in finding monochromatic $\F$-coverings that
are as small as possible; for example, we might want to cover $G$ using as few
monochromatic paths or cycles as possible.

This type of problem goes back to a footnote in a 1967 paper of Gerencsér and
Gyárfás \cite{Gerencser1967} in which it is shown that in every colouring of
the edges of the complete graph $K_n$ with two colours, one can find two
monochromatic paths that form a partition of (and, in particular, a covering) all
the vertices. Over the last fifty years, such problems have been studied in
many variations, including for more than two colours
\cite{EGP1991,Gya1989,Pokrovskiy2014}, for various other choices of $\F$ (most
notably for the family of cycles
\cite{All2008,BT2010,Gyarfas1983,GRSS2006,LRS1998}, but also for regular graphs
\cite{SSS2013}, bounded-degree graphs \cite{GS2016}, trees
\cite{Aha2001,CFG+2012}), and for other choices of $G$ (complete bipartite and
multipartite graphs \cite{CFG+2012,Gya1989,Hax1997,SS2014}, graphs satisfying a
minimum degree condition \cite{BBG+2014,DN2017,Let2015}, random graphs
\cite{BD2017,KMStoappear,KMN+2017}, graphs with bounded independence number
\cite{BBG+2014,Sar2011}, \dots). We note that like the Gerencsér-Gyárfás result
mentioned above, most (but not all) of these results apply to the stronger
situation where one wants to \emph{partition} the vertices of the graph into
disjoint monochromatic pieces (as opposed to just covering the vertices). For
more details we refer to the recent survey of Gyárfás \cite{Gyarfas2016}.

The specific focus of this paper is on monochromatic $\F$-coverings that
altogether \emph{do not use too many different colours}. For a collection $\S$
of monochromatic edge-coloured graphs, we denote by $\col(\S)$ the total number
of different colours used by the graphs in $\S$. Then, given a graph $G$, a
family $\F$, and positive integers $r$ and $s$, we will write $c_{r,s}(G,\F)$
for the smallest number with the property that every $r$-colouring of the edges
of $G$ admits a monochromatic $\F$-covering $\S$ such that $|\S|\leq
c_{r,s}(G,\F)$ and $\col(\S)\leq s$.

For the simplest case where $\F=\F_p$ is the collection of paths, where there
are only two colours, and where $G$ is the complete graph,
Erd\H{o}s and Gyárfás \cite{Erdos1995} proved that
\begin{equation}\label{eq:eg} \sqrt{n}\leq c_{2,1}(K_n,\F_p)\leq
2\sqrt{n}.
\footnote{The quantity
$c_{r,s}(G,\F_p)$ was denoted $\pc_{r,s}(G)$ in the abstract. Henceforth,
we will only use the more flexible notation
$c_{r,s}(G,\F_p)$.}
\end{equation}
It is open which of the two bounds (if any) is correct;
Erd\H{o}s and Gyárfás conjectured that the true value is $\sqrt{n}$.
In any case, we observe that this result is in stark contrast to the
above-mentioned result of Gerencsér and Gyárfás \cite{Gerencser1967}, which implies that
\[ c_{2,2}(K_n,\F_p) = 2,\]
which is a constant independent of $n$. One goal of this project was to see how the result
\eqref{eq:eg} generalizes to to other values of $r$ and $s$.

\subsection{Our results}

In this paper, we restrict ourselves to graphs $G$ with
independence number at most $\alpha>0$. We suppose that $r,s,\alpha$ are
constants and that the size of $G$ tends to infinity. Given $r,s,\alpha$, we
write
\[ c_{r,s,\alpha}(n,\F) = \max_{\substack{|V(G)|=n\\\alpha(G)\leq \alpha}}
c_{r,s}(G,\mathcal F).\]
Thus 
$c_{r,s,\alpha}(n,\F)$ is the minimum integer $k$ such that in every
graph $G$ with independence number at most $\alpha$ and every
$r$-colouring of the edges of $G$, there exists a monochromatic $\F$-covering $\S$ of $G$
of size at most $k$ that satisfies $\col(\S)\leq s$.

To state our results, we must first recall the notion of a Kneser hypergraph.
The Kneser hypergraph $\kg^{(\alpha+1)}(r,r-s)$ is
the $(\alpha+1)$-uniform hypergraph on the vertex set
$\binom{[r]}{r-s} = \{X\subseteq [r]\colon |X|=r-s\}$ where the vertices $X_1,\dotsc,X_{\alpha+1}\in \binom{[r]}{r-s}$ form
a hyperedge if and only if they are pairwise disjoint as subsets of $[r]$. 
A result of Alon, Frankl, and Lovász \cite{AFL1986} states that the chromatic number
of this hypergraph is
\begin{equation}\label{eq:kneserchi}
  \chi(\kg^{(\alpha+1)}(r,r-s)) = \begin{cases}
    1 & \text{if $1\leq s < \alpha r/(\alpha+1)$}\\
    1+s-r+\ceil{(s+1)\alpha^{-1}} & \text{if $\alpha r/(\alpha+1) \leq s
    < r$.}
  \end{cases}
\end{equation}
Note that the range $1\leq s<\alpha
r/(\alpha+1)$ corresponds precisely to the case where $\kg^{(\alpha+1)}(r,r-s)$
has no edges.
The case $\alpha=1$ (which corresponds here to the case
where $G=K_n$) was conjectured by Kneser in 1955 and famously established by
Lovász~\cite{Lov1978} in 1978 using topological methods.

Our first result gives a lower bound on $c_{r,s,\alpha}(n,\F)$. Note that there
are certain trivial cases where $c_{r,s,\alpha}(n,\F)$ is very small simply because the
graphs in $\F$ have many isolated vertices. To give an extreme example, if $\F$
contains for every $n\geq 0$ the graph with $n$ vertices and no edges, then
trivially $c_{r,s,\alpha}(n,\F)=1$. The easiest way to avoid such issues is to
insist that each graph in $\F$ has at most a bounded number of isolated vertices.
In addition to this, we will assume that $\F$ is \emph{$\Delta$-bounded}, that is, that every
graph in $\F$ has maximum degree at most $\Delta$. Then we prove the following
lower bound:

\begin{theorem}[Lower bound]\label{thm:lower}
  Given any positive integers $r,s,\alpha,\Delta,K$ such that $r>s$,
  there exists $c>0$ such that the following holds.
  Let $\F$ be a $\Delta$-bounded family of graphs with at most
  $K$ isolated vertices each.
  Then for every $n\in \NN$, we have
  \[ c_{r,s,\alpha}(n,\F) 
  \geq cn^{1/\chi}, \]
  where $\chi = \chi(\kg^{(\alpha+1)}(r,r-s))$.
\end{theorem}

%
We remark that the conclusion of Theorem \ref{thm:lower} fails when $r=s$;
indeed, there are many situations where $c_{r,r,\alpha}(n,\F)$ is known to be constant.
For example, Gyárfás, Ruszinkó, Sárközy, and Szemerédi \cite{GRSS2006} proved that
$c_{r,r,1}(n,\F_c) \leq 100r\log r$.
Sárközy \cite{Sar2011}
proved that $c_{r,r,\alpha}(n,\F_c)\leq 25(\alpha r)^2\log (\alpha r)$.
Sárközy, Selkow, and Song \cite{SSS2013} proved 
that if $\F$ contains the graph on a single vertex and all connected
$d$-regular graphs, then $c_{r,r,1}(n,\F)\leq 100r\log r+2rd$. 
For more general families, Grinshpun and
Sárközy \cite{GS2016} showed that if $\F$ is $\Delta$-bounded and contains at least
one graph on $i$ vertices for every $i\geq 1$, then $c_{2,2,1}(n,\F) \leq
2^{O(\Delta\log \Delta)}$.

We also prove an upper bound that matches the lower bound given by Theorem
\ref{thm:lower} in many cases. Note again that it is possible to choose $\F$
so that $c_{r,s,\alpha}(n,\F)$ is trivially very large; for example, if $\F$
only contains a single fixed graph then it is obvious that $c_{r,s,\alpha}(n,\F)
= \Omega(n)$. Our way to avoid this kind of problem will be to assume that there is some
$\eps>0$ such that for every $i\geq 1$, $\F$ contains
at least one graph $F$ with $|V(F)|\in [\eps i,i]$. In fact,
our proof (but perhaps not the result) requires the stronger assumption that at least one such graph is \emph{bipartite}.
We prove:

\begin{theorem}[Upper bound]\label{thm:upper}
  Given any positive integers $r,s,\alpha,\Delta$ such that $r>s$, and any $\eps>0$,
  there exists $C>0$ such that the following holds.
  Let $\F$ be a $\Delta$-bounded family $\F$ of graphs such that for every $i\geq
  1$, there is a bipartite $F\in \F$ with $\eps i\leq |V(F)| \leq i$.
  Then for every $n\in \NN$, we have
  \[ c_{r,s,\alpha}(n,\F) \leq Cn^{1/\chi}+c_{r,r,\alpha}(n,\mathcal F), \]
  where $\chi = \chi(\kg^{(\alpha+1)}(r,r-s))$.
\end{theorem}

This upper bound coincides asymptotically with the lower bound given by
Theorem \ref{thm:lower} whenever we know that $c_{r,r,\alpha}(n,\mathcal F) =
O(n^{1/\chi})$. As mentioned above, in many situations it is even known
that $c_{r,r,\alpha}(n,\mathcal F) = O(1)$. We can thus obtain asymptotically
tight results in several different cases. From the above-mentioned result of
Sárközy \cite{Sar2011} we immediately obtain:

\begin{corollary}[Paths and cycles]\label{cor:1}
  Let $r,s,\alpha$ be fixed positive integers such that $r>s$.
  Let $\chi=\chi(\kg^{(\alpha+1)}(r,r-s))$.
  Let $\F_p$ be the family of all paths and $\F_c$ be the family of all cycles.
  Then 
  \[ \Omega(n^{1/\chi})\leq c_{r,s,\alpha}(n,\F_p) \leq c_{r,s,\alpha}(n,\F_c)\leq O(n^{1/\chi}). \]
\end{corollary}

In particular, setting $\alpha = 1$ and using \eqref{eq:kneserchi} gives
\[ c_{r,s}(K_n,\F_p)= \Theta(n^{1/\max{\{1,2+2s-r\}}}) \]
thus generalizing the Erd\H{o}s-Gyárfás result \eqref{eq:eg} to more colours (and the same
holds for $\F_c$ instead of $\F_p$).

Similarly, using the result of Sárközy, Selkow, and Song \cite{SSS2013}, we get
the following result for covering complete graphs by regular graphs:

\begin{corollary}[$d$-regular graphs]\label{cor:2}
  Let $r,s,d$ be fixed positive integers such that $r>s$.
  Let $\chi=\chi(\kg^{(2)}(r,r-s)) = \max{\{1,2+2s-r\}}$.
  Let $\F_d$ be the family containing all connected $d$-regular graphs and also the graph
  with a single vertex and no edges.
  Then
  \[ c_{r,s}(K_n,\F_d) = \Theta(n^{1/\chi}). \]
\end{corollary}

Note that the bounds in Corollaries~\ref{cor:1} and \ref{cor:2}
are only tight up to a large multiplicative
factor depending on $r$, $s$, and $\alpha$ (resp.\ $d$). It would be interesting to
determine these factors more precisely.
As mentioned earlier, even the case where $r=2$ and $s=\alpha=1$ is still open.

It is perhaps interesting to note that the proof of Theorem \ref{thm:upper} does not
actually use the Alon-Frankl-Lovász result \eqref{eq:kneserchi}, but rather
works directly with the definition of $\chi$ as the chromatic number of
$\kg^{(\alpha+1)}(r,r-s)$. On the other hand, our proof of 
Theorem~\ref{thm:lower}
really uses the value of $\chi$ given by \eqref{eq:kneserchi}, or, more precisely,
it uses the lower bound on $\chi$ implied by \eqref{eq:kneserchi},
which is by far the more difficult direction.

\subsection{Notation}

We write $[k]=\{1,\dotsc,k\}$. We write $\binom{A}{\ell}$ for the set of all
$\ell$-element subsets of the set $A$. 
If $G$ is a graph and $V_i,V_j$ are disjoint subsets of the vertices of $G$,
then we denote by $G[V_i,V_j]$ the bipartite subgraph induced by the two parts
$V_i$ and $V_j$, and we write $e_G(V_i,V_j)$ for the number of edges of $G[V_i,V_j]$.

Since we are aiming for asymptotic statements, we routinely omit rounding
brackets whenever they are not essential.

\section{Proof of Theorem~\ref{thm:lower}}

Suppose that we are given positive integers $r,s,\alpha,\Delta,K$ such that $r>s$.
Let $\chi$ denote the chromatic number of $\kg^{(\alpha+1)}(r,r-s)$. 
We need to show
that there is a constant $c=c(r,s,\alpha,\Delta,K)>0$ such that
\[ c_{r,s,\alpha}(n,\F) \geq c n^{1/\chi}  \]
for all $n\in \NN$ and all $\Delta$-bounded families $\F$ of graphs with at most $K$
isolated vertices each.
In other words, we need to construct an $r$-coloured graph $G$ with
independence number at most $\alpha$ such that every monochromatic $\F$-covering $\S$
of $G$ with $\col(\S)\leq s$ has size at least $cn^{1/\chi}$.

The construction will use \emph{Johnson graphs}. The Johnson graph $J(a,b)$ is
the graph with the vertex set $\binom{[a]}{b}$ where two vertices $X$ and $Y$ are
joined by an edge if they have a non-empty intersection (so it is the
complement of the Kneser graph $\kg^{(2)}(a,b)$). Is is easy to see that the
independence number of $J(a,b)$ is at most $\floor{a/b}$: every collection of
$\floor{a/b}+1$ sets in $\binom{[a]}{b}$ covers in total $(\floor{a/b}+1)b>a$
elements, counted with multiplicities, so that at least two of the sets must
intersect. 

To prove Theorem~\ref{thm:lower}, we use
different constructions depending on the parameters. We distinguish between
three cases.

\paragraph{Case 1.} Suppose first that $1\leq s< \alpha r/(\alpha+1)$, i.e.,
that $\chi=1$ by \eqref{eq:kneserchi}. Let $G$ be a blow-up of $J(r,r-s)$ where every vertex is
replaced by a clique on $n/\binom{r}{r-s}$ vertices and where every edge is
replaced by a complete bipartite graph between the corresponding cliques. For a
vertex $X$ of $J(r,r-s)$, we write $V_X$ for the vertices of $G$ in the clique
corresponding to $X$.

Note that $G$ has the same independence number as $J(r,r-s)$, which
is at most $\floor{r/(r-s)}$. The assumption
$\alpha r/(\alpha+1)>s$ implies that
\[ \frac{r}{r-s}< \frac{r}{r-\alpha r/(\alpha+1)}= \alpha+1 \]
and so the independence number of $G$ is at most $\alpha$,
as required.

We now colour the edges of $G$ with colours from $[r]$ as follows. Let $uv$
be an edge of $G$. Then there exist vertices $X$ and $Y$ of $J(r,r-s)$ such
that $u\in V_X$ and $v\in V_Y$. Moreover, we either have $X=Y$, or $\{X,Y\}$ is an
edge in $J(r,r-s)$, and in both cases, $X\cap Y\neq \emptyset$. We then
colour $uv$ with any colour belonging to the set $X\cap Y\subseteq [r]$.

Finally, suppose that $\S$ is a monochromatic $\F$-covering of $G$ such that
$\col(\S)\leq s$.
Then there is some $X\subseteq
[r]$ of size $r-s$ that is disjoint from the set of colors used by
the graphs in $\S$. By our choice
of colouring, all edges touching $V_X$ have a colour in $X$, so the vertices in
$V_X$ can only be covered using isolated vertices. Since every graph in $\S$
has at most $K$ isolated vertices, this means that
$|\S| \geq |V_X|/K \geq n/(K\binom{r}{r-s})$, completing the proof in this case
(since $\chi=1$).

\paragraph{Case 2.} Suppose now that $s\geq \alpha r/(\alpha+1)$ and assume
additionally that $s< \chi \alpha$. Then by \eqref{eq:kneserchi}, we have
$\chi = 1+s-r+\ceil{(s+1)\alpha^{-1}} \leq r$, where the last inequality
follows from $s+1\leq r$. Since additionally $s+1\leq \chi\alpha$,
we can fix integers $1\leq k_1,\dotsc,k_\chi\leq \alpha$ such that $k:=
k_1+\dotsb+k_\chi \in \{s+1,\dotsc,r\}$.

We now construct an $n$-vertex graph $G$ as follows. We start with a blow-up of
the complete graph $K_\chi$ where the $i$-th vertex is replaced by a set
$V_i$ of $n^{i/\chi}$ vertices, except for the $\chi$-th vertex, which is
replaced by a set set $V_\chi$ of \[ |V_\chi| = n - n^{1/\chi} - n^{2/\chi} -
\dotsb - n^{(\chi-1)/\chi} \geq n-o(n)\] vertices. Each edge $ij$ of $K_\chi$
is replaced by a complete bipartite graph between the corresponding sets
$V_i$ and $V_j$. We further partition each set $V_i$ equitably into $k_i$
parts $V_{i,1},\dotsc,V_{i,k_i}$, and insert all edges where both
endpoints are contained in the same set $V_{i,j}$. Thus for each $i$, the
graph $G[V_i]$ is the disjoint union of $k_i$ cliques of size
$|V_i|/k_i$. This defines the graph $G$. It is easy to see that $G$ has
independence number $\max{\{k_i: 1\leq i\leq \chi\}}\leq\alpha$.

Next, we colour the edges of $G$ as follows. First, we fix an arbitrary
bijection
\[ \phi\colon \{(i,j): 1\leq i\leq \chi \text{ and }1\leq
j\leq k_i\}\to [k].\]
Such a bijection exists because $k_1+\dotsb+k_\chi=k$.
Then we distinguish two cases. If $uv$ is an edge of $G$ with both endpoints
in the same set $V_{i,j}$, then $uv$ receives the colour $\phi(i,j)$. On the
other hand, if $uv$ goes between the sets $V_{i,j}$ and $V_{i',j'}$ where
$i<i'$, then we $uv$ receives the colour $\phi(i,j)$. Note that by construction,
there are no edges going between to sets $V_{i,j}$ and $V_{i,j'}$ for $j\neq j'$.
Since $k\leq r$, this is a colouring with at most $r$ colours.

Now suppose that $\S$ is a monochromatic $\F$-covering of $G$ such that
$\col(\S)\leq s$. Since $s<k$, there is then some pair $(i,j)$ with
$1\leq i\leq \chi$ and $1\leq j\leq k_i$ such that $\phi(i,j)$ is not
the colour of any graph in $\S$.
Now observe that the only edges incident to $V_{i,j}$
that do not use the colour $\phi(i,j)$ are those that have an endpoint
in $V_1\cup\dotsb\cup V_{i-1}$.
In particular, every graph in $\S$, having maximum degree at most $\Delta$
and at most $K$ isolated vertices, can cover at most
$\Delta(|V_1|+\dotsb+|V_{i-1}|)+K$ vertices of $V_{i,j}$.
Now $|V_{i,j}| \geq n^{i/\chi}/r$ implies
\[ \begin{split}
  \Delta(|V_1|+\dotsb+|V_{i-1}|)+K & = \Delta(n^{1/\chi} + \dotsb + n^{(i-1)/\chi})+K\\
  & \leq (1+o(1))\cdot (\Delta+K)\cdot n^{(i-1)/\chi}\\
& \leq (1+o(1)) \cdot r(\Delta +K) \cdot n^{-1/\chi} |V_{i,j}|, \end{split}\]
and so to cover $V_{i,j}$ completely, $\S$ must contain
at least
$(1-o(1))n^{1/\chi}/(r(\Delta+K))$ graphs, completing the proof in this case.

\paragraph{Case 3.} Finally, assume $s\geq \alpha r/(\alpha+1)$ and $s\geq \chi
\alpha$.
The construction in this case is a combination of the constructions
used in the two previous cases. We will construct a graph $G$ on $n$ vertices
as follows.
As in Case 2, we start with a blow-up of
the complete graph $K_\chi$ where the $i$-th vertex is replaced by a set $V_i$
of $|V_i|=n^{i/\chi}$ vertices, except for the last vertex, which is replaced
by a set $V_\chi$ of \[ |V_\chi| = n - n^{1/\chi} - n^{2/\chi} - \dotsb -
n^{(\chi-1)/\chi} \geq n-o(n)\] vertices. Each edge $ij$ of $K_\chi$ is
replaced by a complete bipartite graph between the corresponding sets $V_i$ and
$V_j$. This defines the edges going between different sets $V_i$ and $V_j$.

Next, we specify what each graph $G[V_i]$ looks like. For $G[V_1]$, we use a
similar construction as in Case 1. Let
$t:= r-\alpha(\chi-1)$ and note that since $s\geq \chi \alpha >
\alpha(\chi-1)$, we have $t> r-s$. We let $G[V_1]$ be a blow-up of the
Johnson graph $J(t,r-s)$ where
every vertex is replaced by a clique on $|V_1|/\binom{t}{r-s}$ vertices, and
where every edge is replaced by a complete bipartite graph between the
corresponding cliques. For later reference, we define $V_{1,X}\subseteq V_1$
to be the vertex set of the clique corresponding to the vertex $X$ of
$J(t,r-s)$. For $1< i \leq \chi$, we let $G[V_i]$ be the union of $\alpha$
vertex-disjoint cliques of size $|V_i|/\alpha$, somewhat similarly as in Case 2.
We will write
$V_{i,1},\dotsc,V_{i,\alpha}\subseteq V_i$ for the vertex sets of these
cliques. This completes the definition of $G$.

We first check that $G$ really has independence number at most $\alpha$.
It is immediate from the construction that $\alpha(G) = \max{\{\alpha(G[V_i]):
1\leq i \leq \chi\}}$. Moreover, it is easy to see that for $i>1$, we have
$\alpha(G[V_i]) = \alpha$. So it remains only to consider $i=1$.
Observe that $G[V_1]$ has the same independence number as $J(t,r-s)$, which
is at most $\floor{t/(r-s)}$. It is thus sufficient to prove that
$t/(r-s) < \alpha+1$, which is easily seen to be true using the definition
of $\chi$. Indeed, since $t = r-\alpha(\chi-1)$, the inequality
$t < (\alpha+1)(r-s)$ is equivalent to
\[ s<\alpha(r-s+\chi-1), \]
which is true because $r-s+\chi-1 = \ceil{(s+1)/\alpha}$ using
\eqref{eq:kneserchi} and the assumption $s\geq \alpha r/(\alpha+1)$.
Hence we have $\alpha(G)\leq \alpha$, as required.

We now define a colouring of the edges of $G$ with $r$ colours, where we
distinguish several cases. First, suppose that $uv$ is an edge with $u,v\in
V_1$. Then there exist vertices $X,Y$ of $J(t,r-s)$ such
that $u\in V_{1,X}$ and $v\in V_{1,Y}$; moreover, for these $X,Y$ it holds
that $X\cap Y \neq \emptyset$ (they are either identical or represent an edge
in $J(t,r-s)$). We then
colour $uv$ with any colour in $X\cap Y$. Second, assume that $uv$ has
exactly one endpoint (say, $u$) in $V_1$ and the other in $V_i$ for some
$i>1$. Then there is some vertex $X$ of $J(t,r-s)$ such
that $u\in V_{1,X}$, and we colour $uv$ with any colour in $X$.
Lastly, to colour the remaining edges, fix any bijection
\[ \phi\colon
\{(i,j) : 1< i \leq \chi\text{ and } 1\leq j\leq \alpha\}
\to [r]\setminus[t]. \]
Such a bijection exists because $r-t = \alpha(\chi-1)$.
If $uv$ is an edge with both endpoints in the same set $V_i$
for $i>1$, say $u,v\in V_{i,j}$, then we colour $uv$ with the colour
$\phi(i,j)$ (note that there are no edges between $V_{i,j}$ and
$V_{i,j'}$ for $j\neq j'$). If $uv$ is an edge going between $u\in V_{i,j}$
and $v\in V_{i',j'}$ where $i<i'$, then we colour $uv$ with the colour
$\phi(i,j)$. Thus we have coloured all the edges.

We make two observations at this point:
\begin{enumerate}[(i)]
 \item Every edge incident to $V_{1,X}$ is coloured with a colour from $X$,
   for every vertex $X$ of $J(t,r-s)$;
 \item For every $1<i\leq \chi$ and $1\leq j\leq \alpha$, the only edges
   incident to $V_{i,j}$ that do not use the colour $\phi(i,j)$ are those
   that are incident to a set $V_{i'}$
   where $i'<i$. In particular, every monochromatic copy of a graph $F\in \F$ 
   that uses a colour different from $\phi(i,j)$ can cover at most
   \[ \begin{split} \Delta (|V_1|+\dotsb+|V_{i-1}|)+K
     &\leq \Delta(n^{1/\chi} + \dotsb + n^{(i-1)/\chi})+K\\
     &\leq (1+o(1))\cdot (\Delta+K)\cdot n^{(i-1)/\chi}\\
     & \leq (1+o(1))\cdot \alpha (\Delta+K)\cdot n^{-1/\chi}
   |V_{i,j}|\end{split}\]
   vertices of $V_{i,j}$, where we use that $F$ has maximum degree at most $\Delta$
   and at most $K$ isolated vertices.
\end{enumerate}

To complete the proof, suppose that $\S$ is a monochromatic
$\F$-covering of $G$ such that $\col(\S)\leq s$. Denoting by
$\Col(\S)$ the set of all colours used by graphs in $\S$,
we distinguish two possible cases.

The first case is when $\Col(\S)$ contains at most $t-(r-s)$ colours
from $[t]$. In this case, there is some set $X$
of $r-s$ colours in $[t]$ that do not
belong to $\Col(\S)$. But then, as all edges incident to $V_{1,X}$
use a colour from $X$ (see (i)), the only way in which $\S$ can cover the
vertices in $V_{1,X}$ is by using isolated vertices. Since each graph in 
$\S$ has at most $K$ isolated vertices, this implies
$|\S|\geq |V_{1,X}|/K \geq 
n^{1/\chi}/K$, completing the proof in this case.

In the other case, $\Col(\S)$ contains at least $t-(r-s)+1$
colours from $[t]$. Since $\col(\S)\leq s$,
this means that at most $s-t+(r-s)-1=r-t-1$ colours from $\Col(\S)$
can be contained in $[r]\setminus [t]$. In particular, there is a colour
$a\in [r]\setminus [t]$ that is not used by any of the graphs in $\S$.
Let $(i,j)= \phi^{-1}(a)$ and consider the set
$V_{i,j}$. Then by (ii), every graph in $\S$ can cover at most $(1+o(1))
\cdot \alpha(\Delta+K)\cdot n^{-1/\chi}|V_{i,j}|$ vertices of $V_{i,j}$,
so $|\S|\geq (1-o(1))\cdot n^{1/\chi}/(\alpha(\Delta+K))$. This completes the
proof of Theorem \ref{thm:lower}.

\section{Proof of Theorem~\ref{thm:upper}}

Let $r,s,\alpha$ be positive integers with $r>s$. Let $\K:=
\kg^{(\alpha+1)}(r,r-s)$ and $\chi:= \chi(\K)$. Let $G$ be a graph on $n$
vertices with independence number at most $\alpha$, and suppose that the edges
of $G$ are coloured with $r$ colours, which we assume to come from the set
$[r]=\{1,\dotsc,r\}$. Then the vertices of $\K$ correspond naturally to sets of $r-s$
colours. Let $\Delta,\eps>0$ and let $\F$ be a $\Delta$-bounded family of
graphs with such that for every $i\geq 1$, $\F$ contains at least one bipartite
graph with at least $\eps i$ and at most $i$ vertices. In particular, $\F$
contains the graph on a single vertex and with no edges. We will show that
there is a monochromatic $\F$-covering $\S$ of $G$ such that
\[ |\S|\leq Cn^{1/\chi}+ c_{r,r}(G,\F) \text{ and } \col(\S) \leq s,\]
where $C = C(r,s,\alpha,\eps)>0$ is a suitable constant.

We first note that if $s<\alpha r/(\alpha+1)$, then by \eqref{eq:kneserchi}, we
have $\chi=1$. In this case, we can simply cover $G$ by $n$ single vertices,
and we are done. Therefore, we will assume from now on that $s\geq\alpha
r/(\alpha+1)$.

We start by introducing some notation. If $\S$ is a monochromatic $\F$-covering of
$G$ and $X\in V(\K)$ is a set of $r-s$ colours, then we write
$V_{\S,X}\subseteq V(G)$ for the set of all vertices of $G$ that are covered in
$\S$ exclusively by graphs having a colour in $X$, that is,
\[ V_{\S,X} := \{v\in V(G) \colon \text{every $H\in\S$
such that $v\in V(H)$
has a colour in $X$}\}. \]
Note that
$\S\subseteq \S'$ implies $V_{\S',X}\subseteq V_{\S,X}$
for all $X\in V(\K)$: adding more graphs to $\S$ can never increase
one of the sets $V_{\S,X}$. Our goal will be to construct a small
monochromatic $\F$-covering $\S$ such that $V_{\S,X}=\emptyset$ for some $X\in
V(\K)$. Note that in this case, $G$ is completely covered by the graphs in $\S$
that have colours not in $X$, so by removing all graphs with a colour in $X$ from $\S$,
we can obtain a monochromatic $\F$-covering $\S'\subseteq \S$
with $\col(\S')\leq s$.

With this goal in mind, we define a quantity to track the sizes of the sets $|V_{\S,X}|$:
\[ \delta(\S) :=
\sum_{X\in V(\K)}\log{|V_{\S,X}|}, \]
where we can set $\delta(\S)=-\infty$ if $|V_{\S,X}|=0$ holds for some $X\in
V(\K)$. Note that since $|V_{\S,X}|\leq n$, we always have
the bound
$\delta(\S)\leq \binom{r}{r-s}\log n$. Our central
claim is:

\begin{claim}\label{cl:main}
  There is a constant $\beta>0$ such that the following holds.
  If $\S$ is a monochromatic $\F$-covering of $G$ such that
  $|V_{\S,X}|> n^{1/\chi}$ for all $X\in V(\K)$, then $G$ contains a (nonempty)
  collection $\mathcal H=\{H_1,\dotsc,H_t\}$ of monochromatic copies of graphs
  in $\F$
  such that
  \begin{equation}\label{eq:dec}
    \delta(\S)-\delta(\S\cup \mathcal H) \geq \beta tn^{-1/\chi}\log n.
  \end{equation}
\end{claim}

We postpone the proof of this claim and first show how it can serve to imply
the theorem. 
We construct a monochromatic $\F$-covering step by step,
starting with some monochromatic $\F$-covering $\S_0$ of size $c_{r,r}(G,\F)$
(which exists by definition). Then as long as $|V_{\S_i,X}|> n^{1/\chi}$ for
all $X\in V(\K)$, we construct $\S_{i+1}$ from $\S_i$ by setting
$\S_{i+1}=\S_i\cup \mathcal H$ for a collection $\mathcal H$ as given by Claim
\ref{cl:main}. Note that since $\delta(\S_0) \leq \binom{r}{r-s}\log n$, and
since $\delta(\S) \leq 0$ implies that $|V_{\S,X}|\leq 1\leq n^{1/\chi}$ for
some $X\in V(\K)$, it follows from \eqref{eq:dec} that this
process must end after adding at most
$\binom{r}{r-s}n^{1/\chi}/\beta$ graphs to $\S_0$. In other words, we end up
with a monochromatic $\F$-covering $\S^*$ of size $|\S^*|\leq
c_{r,r}(G,\F) + \binom{r}{r-s}n^{1/\chi}/\beta$ such that $|V_{\S^*,X}| \leq
n^{1/\chi}$ holds for at least one $X\in V(\K)$. From this we obtain another
monochromatic $\F$-covering
$\S$ by adding to $\S^*$ at most $n^{1/\chi}$ single-vertex
graphs covering the vertices in $V_{\S^*,X}$. Note that then $V_{\S,X} =
\emptyset$ and $|\S|\leq c_{r,r}(G,\F) + \binom{r}{r-s}n^{1/\chi}/\beta +
n^{1/\chi}$. As mentioned above, we can then find a monochromatic $\F$-covering
$\S'\subseteq \S$ with $\col(\S')\leq s$, completing the proof of the theorem.

\subsection{Proof of Claim \ref{cl:main}}

It remains to give the proof of Claim \ref{cl:main}. The proof will use the
following lemma, whose proof we omit (it is a standard application of
Szemerédi's regularity lemma, see for example \cite[Theorem 2.1]{KS1996}).

\begin{lemma}\label{lemma:embedding}
  For every $\eps>0$ and $\Delta>0$ there is a constant $\delta>0$ such that
  the following holds for all sufficiently large $n$.
  If $G = (A,B,E)$ is a bipartite graph with $|A|=|B|=n$ and $|E|\geq \eps n^2$,
  then it contains as a subgraph every bipartite graph with maximum degree at most $\Delta$
  and at most $\delta n$ vertices.
\end{lemma}

In the following, let $\S$ be a monochromatic $\F$-covering of $G$ such that $|V_{\S,X}|>
n^{1/\chi}$ for all $X\in V(\K)$. We first show:

\begin{claim}
There exists a hyperedge
$\E=\{X_1,\dotsc,X_{\alpha+1}\}$ of $\K$ such that
\begin{equation}\label{eq:ratios} n^{-1/\chi}\leq
\frac{|V_{\S,X_i}|}{|V_{\S,X_j}|}\leq n^{1/\chi} \quad \text{for all $i,j\in
[\alpha+1]$.} \end{equation}
\end{claim}

\begin{proof}
  Fix any $c>1$ and let $b\in (n^{1/\chi},c n^{1/\chi})$ be such that
  $b\leq |V_{\S,X}|$ holds for
  all $X\in V(\K)$. This is possible because we assume that
  $|V_{\S,X}|>n^{1/\chi}$ for all $X\in V(\K)$. Then, because $b \leq
  |V_{\S,X}| \leq n$, the map $X\mapsto \floor{\log_b{|V_{\S,X}|}}$ assigns
  each vertex of $\K$ a number between $1$ and $\floor{\log_b n} \leq \chi-1$.
  Hence, by definition of the chromatic number, there is a hyperedge $\E =
  \{X_1,\dotsc,X_{\alpha+1}\}$ in which all vertices receive the same number.
  Then for all $i,j\in [\alpha+1]$, we have
  \[ -1< \log_b |V_{\S,X_i}| - \log_b |V_{\S,X_j}| < 1, \] so $n^{-1/\chi}/c <
  |V_{\S,X_i}|/|V_{\S,X_j}| < cn^{1/\chi}$. Since $c$ can be arbitrarily close
  to $1$, and as $\K$ is finite, this implies the claim.
\end{proof}
 
Let now $\E=\{X_1,\dotsc,X_{\alpha+1}\}$ be a hyperedge of $\K$ satisfying
\eqref{eq:ratios}. We will assume the elements of $\E$ are ordered so that \[
  |V_{\S,X_1}|\geq |V_{\S,X_2}|\geq \dotsb \geq |V_{\S,X_{\alpha+1}}|. \]

\begin{definition}[Removable set]
  Let us say that a subset $W\subseteq V_{\S,X_i}$ is \emph{removable} if $G$
  contains a monochromatic copy $H$ of some graph in $\F$ such that
  (i) the colour of $H$ is in $[r]\setminus
  X_i$ and (ii) $W\subseteq V(H)$.
\end{definition}

The idea behind this definition is that if $W\subseteq V_{\S,X_i}$ is
removable, then by adding the graph $H$ to $\S$, we can decrease the size of
$|V_{\S,X_i}|$ by at least $|W|$: indeed, recalling the definition of
$V_{\S,X_i}$, we see that $V_{\S\cup \{H\},X_i} \subseteq V_{\S,X_i}\setminus
W$.

\begin{claim}\label{cl:aux2}
  There is a constant $C>0$ and some $i\in [\alpha+1]$ such that the following
  holds: There exist
  $t\leq C|V_{\S,X_i}|/|V_{\S,X_{\alpha+1}}|$
  disjoint removable sets
  $W_1,\dotsc,W_t\subseteq V_{\S,X_i}$ 
  covering all except for at most
  $|V_{\S,X_{\alpha+1}}|/2$ vertices in $|V_{\S,X_i}|$.
\end{claim}
\begin{proof}
  Observe first that it is enough to show the following statement:
  for every choice of subsets $V_1,\dotsc,V_{\alpha+1}$ where
  $V_i\subseteq V_{\S,X_i}$ and where each $V_i$ has size
  $|V_{\S,X_{\alpha+1}}|/2$, there is some $i\in [\alpha+1]$ and a subset
  $W\subseteq V_i$ of size at least $|V_{\S,X_{\alpha+1}}|/C$ that is
  removable. Indeed, we can then repeatedly apply this statement until we have
  covered all but $|V_{\S,X_{\alpha+1}}|/2$ vertices in at least one set
  $V_{\S,X_i}$, and it is clear that this requires at most  $C|V_{\S,X_i}|/
  |V_{\S,X_{\alpha+1}}|$ subsets of $V_{\S,X_i}$.
  So we will now
  prove this other statement instead.

  Fix sets $V_1,\dotsc,V_{\alpha+1}$ as above. For brevity, write $\eta :=
  |V_{\S,X_{\alpha+1}}|/2 = |V_1| = \dotsb = |V_{\alpha+1}|$.
  From the fact that $G$ has independence number at most $\alpha$ it follows
  that there exist distinct $i,j\in [\alpha+1]$ such that $e_G(V_i,V_j)\geq
  \eta^2/(\alpha+1)^2$. This can be seen by simple double counting: for every
  choice of $\alpha+1$ vertices $v_i\in V_i$ for $i\in [\alpha+1]$, there
  must be two vertices that are connected by an edge. Going over all ways to
  choose such vertices, we thus obtain $\eta^{\alpha+1}$ edges, where every
  edge is obtained at most $\eta^{\alpha-1}$ times; so there must be
  $\eta^2$ edges going between the sets $V_1,\dotsc,V_{\alpha+1}$. In
  particular, for some $i\neq j$, we have $e_G(V_i,V_j)\geq
  \eta^2/(\alpha+1)^2$.

  Suppose now that $e_G(V_i,V_j)\geq \eta^2/(\alpha+1)^2$.
  Let $k\in [r]$ denote the majority colour of the edges
  in $G[V_i,V_j]$ and write $G_k[V_i,V_j]$ for the subgraph consisting only of
  the edges having colour $k$. Then it is clear that $G_k[V_i,V_j]$ has at
  least $\eta^2/(r(\alpha+1)^2)$ edges.

  Recall that we assume that $\F$ is $\Delta$-bounded and that there is some
  $\eps>0$ such that for every $n'\geq 1$, the family $\F$ contains at least
  one bipartite subgraph $F\in \F$ with $\eps n' \leq |V(F)|\leq n'$.

  Applying Lemma \ref{lemma:embedding} to $G_k[V_i,V_j]$ (which is
  possible for large $n$ since $|V_i|=|V_j|=\eta>n^{1/\chi}/2$) , we obtain that
  $G_k[V_i,V_j]$ contains as a subgraph every $\Delta$-bounded bipartite graph on at
  most $2(\Delta+1) \eta/(C\eps)$ vertices, for some sufficiently large constant $C>0$.
  In particular, $G_k[V_i,V_j]$ contains a copy of a graph $F\in \F$ with at
  least $2(\Delta+1) \eta/C$ vertices. In fact, since $F$ has maximum degree at most
  $\Delta$, it can be embedded in such a way that is uses at least $2\eta/C$
  vertices of $V_i$ and at least $2\eta/C$ vertices of $V_j$ (for every $\Delta$ non-isolated
  vertices in $V_i$ we must embed at least one vertex in $V_j$, whereas the isolated vertices
  can be embedded arbitrarily). Denote this copy by $H$ and note that as a subgraph of
  $G_k[V_i,V_j]$ it is clearly monochromatic in colour $k$.
  Since the sets $X_i$ and $X_j$ are disjoint (they are
  part of a hyperedge in $\K$), they cannot both contain $k$, and so at least one
  of the sets $V(H)\cap V_i$ or $V(H)\cap V_j$ is removable, and both these sets
  have size $2\eta/C =|V_{\S,X_{\alpha+1}}|/C$.
\end{proof}

Let $W_1,\dotsc,W_t\subseteq V_{\S,X_i}$ be disjoint removable sets  as in
Claim \ref{cl:aux2} and let $\mathcal H = \{H_1,\dotsc H_t\}$ be the
corresponding collection of subgraphs, so that $H_j$ is a monochromatic copy of a
graph in $\F$ that covers $W_j$ and uses a colour outside $X_i$.
By Claim \ref{cl:aux2} and the definition of \emph{removable}, we have
$|V_{\S\cup \mathcal H,X_i}| \leq |V_{\S,X_{\alpha+1}}|/2
< |V_{\S,X_i}|$. This implies 
immediately that the collection $\mathcal H$ is nonempty. It also implies that
\[ \begin{split}
  \delta(\S\cup \mathcal H)
  &= \sum_{j\in [\alpha+1]}
  \log{|V_{\S\cup \mathcal H,X_j}|}\\
  &\leq \sum_{j\in [\alpha+1]\setminus \{i\}}
  \log{|V_{\S,X_j}|}+\log{|V_{\S\cup \mathcal H,X_i}|}\\
  &\leq \delta(\S) -\log{|V_{\S,X_i}|}   +\log (|V_{\S,X_{\alpha+1}}|/2)\\
  &= \delta(\S)- \log (2|V_{\S,X_i}|/|V_{\S,X_{\alpha+1}}|),
\end{split}
\]
and so
\[ \delta(\S)-\delta(\S\cup \mathcal H)\geq \log (2|V_{\S,X_i}|/|V_{\S,X_{\alpha+1}}|). \]
At the same time, 
using $1\leq t\leq C|V_{\S,X_i}|/|V_{\S,X_{\alpha+1}}|$ and
$|V_{\S,X_i}|/|V_{\S,X_{\alpha+1}}|\leq n^{1/\chi}$,
we get
\[ 
\frac{\log (2|V_{\S,X_i}|/|V_{\S,X_{\alpha+1}}|)}{t}
\geq \frac{\log
(2|V_{\S,X_i}|/|V_{\S,X_{\alpha+1}}|)}{C|V_{\S,X_i}|/|V_{\S,X_{\alpha+1}}|}
\geq \frac{\log
(2n^{1/\chi})}{Cn^{1/\chi}} \geq
\frac{n^{-1/\chi}\log n}{C\chi},\]
so 
\[ \delta(\S)-\delta(\S\cup \mathcal H)\geq
\frac{tn^{-1/\chi}\log n}{C\chi},\]
completing the proof of Claim \ref{cl:main}.

\end{document}